\documentclass[11pt]{amsart}
\usepackage[usenames]{color}
\usepackage{fullpage}
\usepackage{amssymb, latexsym}

\theoremstyle{plain}
\newtheorem{thm}{Theorem}[section]
\newtheorem{theorem}[thm]{Theorem}

\newtheorem{lemma}[thm]{Lemma}

\newtheorem{proposition}[thm]{Proposition}

\theoremstyle{definition}

\newcommand{\Z}{\mathbb{Z}}

\newcommand{\AG}[1]{\mathrm{AGG}#1}

\newcommand{\ld}{\backslash}

\numberwithin{equation}{section}

\begin{document}

\title{Linear representation of Abel-Grassmann groups}

\author{David Stanovsk\'y}

\address{Department of Information Systems and Mathematical Modeling, International IT University, Manas st. 34, 050040 Almaty, Kazakhstan}
\address{Department of Algebra, Faculty of Mathematics and Physics, Charles University, Sokolovsk\'a 83, 18675 Praha 8, Czech Republic}

\email{david.stanovsky@gmail.com}

\thanks{Partly supported by the GA\v CR grant 13-01832S}


\date{\today}

\begin{abstract}
We describe an linear representation for Abel-Grassmann groups. As a consequence, we obtain or improve many previous results. In particular, enumeration of Abel-Grassmann groups up to isomorphism is obtained for orders $<512$.
\end{abstract}

\maketitle

\section{Introduction}\label{sec:intro}

By a \emph{groupoid} we mean an algebraic structure $(G,\cdot)$ with a single binary operation.
A groupoid is called \emph{Abel-Grassmann} (shortly, \emph{AG-groupoid}) if it satisfies the identity
$$(a\cdot b)\cdot c=(c\cdot b)\cdot a$$
for every $a,b,c\in G$. Groupoids satisfying an additional identity
$$a\cdot(b\cdot c)=b\cdot(a\cdot c)$$
for every $a,b,c\in G$ are called \emph{AG$^{**}$-groupoids}.
In a number of recent papers, e.g., \cite{DG1,DG2,KS,MK3,P1}, it has been demonstrated that the theory of AG$^{**}$-groupoids has a strong parallel to the theory of commutative semigroups, and this phenomenon seems to be the main motivation for their study. Some indication that this is not a coincidence can be found in \cite{ARS} and \cite[Chapter 3]{JK}, see our Section 4 for comments.

As a proof of concept, in this paper, we focus on a particular subclass of AG$^{**}$-groupoids, called \emph{Abel-Grassmann groups} (shortly, \emph{AG-groups})\footnote{In my opinion, the choice of name is somewhat unlucky, since these structures are not groups. But let me keep the established terminology.}. They are defined as groupoids $(G,\cdot)$ satisfying three axioms similar to the axioms of groups:
\begin{itemize}
	\item the Abel-Grassmann identity $(a\cdot b)\cdot c=(c\cdot b)\cdot a$ holds for every $a,b,c\in G$;
	\item there exists a \emph{left unit} $e$, i.e., an element satisfying $e\cdot a=a$ for every $a\in G$;
	\item for every $a\in G$ there exists a unique $a^{-1}$ such that $a\cdot a^{-1}=a^{-1}\cdot a=e$.
\end{itemize}
For example, every abelian group $A$ is an AG-group, and also $(A,*)$ with $a*b=b-a$ is an AG-group. 
We believe \cite{K,MK1,P2,SA,SGS} is the complete list of references to AG-groups (in earlier papers, they were called LA-groups). 

It has been noticed (e.g., in \cite{SA}) that every AG-groupoid is \emph{medial}, i.e., the identity $(ab)(cd)=(ac)(bd)$ holds for every $a,b,c,d$, and that every AG$^{**}$-groupoid is \emph{paramedial}, i.e., the identity $(ab)(cd)=(db)(ca)$ holds for every $a,b,c,d$ (paramediality is somewhat lesser known, see \cite{CJK} for an account on paramediality). Also, every AG-group is an AG$^{**}$-groupoid. Importantly, AG-groups are \emph{quasigroups} (latin squares), i.e., they possess unique left and right division. While this fact is implicit in \cite{SA}, we could not find a complete proof anywhere. 

The main result of our paper is the following characterization of Abel-Grassmann groups. 

\begin{theorem}\label{Th:main}
The following conditions are equivalent for a groupoid $(G,\cdot)$:
\begin{enumerate}
	\item $(G,\cdot)$ is an AG-group.
	\item $(G,\cdot)$ is an AG$^{**}$-quasigroup.
	\item $(G,\cdot)$ is a paramedial quasigroup with a left unit.
	\item There exists an abelian group $(G,+)$ and its automorphism $\varphi$ satisfying $\varphi^2=id$ such that $a\cdot b=\varphi(a)+b$ for every $a,b\in G$.
\end{enumerate}
\end{theorem}

In particular, condition (4) provides a complete classification of AG-groups in terms of abelian groups and their involutory automorphisms. (For the examples above, take $\varphi(a)=a$, and $\varphi(a)=-a$, respectively.) It can also be interpreted as a \emph{linear representation} for every AG-group over a module over the ring $\Z[x]/(x^2-1)$.
This fact has a number of consquences and easily explains virtually all the results on AG-groups in the papers listed above.

Let us note that Theorem \ref{Th:main} is not surprizing. By an \emph{affine representation} of a quasigroup $(G,\cdot)$ we mean an abelian group $(G,+)$, its automorphisms $\varphi,\psi$, and $c\in G$ such that $$a\cdot b=\varphi(a)+\psi(b)+c$$ for every $a,b\in G$.
Since AG-groups are medial quasigroups, the Toyoda-Bruck representation theorem \cite{T} applies, and we obtain an affine representation with an additional condition that $\varphi\psi=\psi\varphi$. 
Since AG-groups are paramedial quasigroups, the Kepka-N\v emec representation theorem \cite{KN} applies, and we obtain an affine representation with an additional condition that $\varphi^2=\psi^2$.
Existence of a left unit then implies $\varphi^2=\psi=id$.
Nevertheless, in our paper we present a direct proof, which allows a simpler (linear, with $c=0$) representation, and a bit more.

In fact, we prove a stronger version of Theorem \ref{Th:main} in Section 2: there is a \emph{term equivalence} between the varieties of AG-groups, AG$^{**}$-quasigroups, paramedial quasigroups with a left unit, and modules over the ring $\Z[x]/(x^2-1)$. It means, the four varieties are essentially identical, up to choice of the basic operations. Consequently, all properties determined by term operations translate straightforwardly from one setting to another. For example, in modules, congruences (i.e., kernels of homomorphisms) are determined by subalgebras, thus the same correspondence exists in all four settings. At the end of Section 2, we make a comparison to a correspondence developed in \cite{P2}.

In section 3, we enumerate AG-groups up to isomorphism, for every size $<512$, and every size $p_1^{d_1}\cdots p_k^{d_k}$ with all $d_i\leq2$. This dramatically improves results of \cite{SGS}, where enumeration was given up to size 11.

In section 4, we present a few remarks towards extending the results to other classes of AG$^{**}$-groupoids.

We will use the following notation to reduce the number of parentheses: for instance, $ab\cdot cd$ will stand for $(a\cdot b)\cdot(c\cdot d)$, and similarly for other expressions. All identities are meant to be universally quantified, unless stated otherwise.

\section{Equivalence of the four concepts}

We start with a few observations. Most of them can be traced in \cite{P2,SA}.

\begin{lemma}\label{L:paramedial}\ 
\begin{enumerate}
	\item An AG$^{**}$-quasigroup is paramedial, the mapping $f(a)=a/a$ is constant, and $e=a/a$ is a left unit.
	\item A paramedial groupoid with a left unit is an AG-groupoid.
	\item An AG-groupoid with a left unit is an AG$^{**}$-groupoid.
\end{enumerate}
\end{lemma}

\begin{proof}
(1) For paramediality, calculate $$ab\cdot cd=c \cdot(ab\cdot d)=c \cdot(db\cdot a)=db\cdot ca.$$ 
To prove that $a/a=b/b$ for every $a,b$, calculate
$$(a/a)\cdot b=(a/a)\cdot((b/a)a)=(b/a)\cdot((a/a)a)=(b/a)\cdot a=b$$
and divide by $b$ from right. In particular, we see that $e=a/a$ is a left unit.

(2) Let $e$ be a left unit. Then $$ab\cdot c=ab\cdot ec=cb\cdot ea=cb\cdot a.$$

(3) Let $e$ be a left unit again. Then $$a\cdot bc=ea\cdot bc=(bc\cdot a)e=(ac\cdot b)e=eb\cdot ac=b\cdot ac.$$
\end{proof}

As an immediate corollary, we obtain the equivalence $(2)\Leftrightarrow(3)$ of Theorem \ref{Th:main}.

From universal algebraic perspective, it is convenient to work with varieties. Therefore, we have to introduce symbols for all operations we want to perform. We will consider the following varieties.

\emph{The variety of AG-groups} is the variety of algebras $(G,\cdot,^{-1},e)$ satisfying the identities
$$ab\cdot c=cb\cdot a,\quad aa^{-1}=a^{-1}a=e,\quad ea=a.$$

\emph{The variety of AG$^{**}$-quasigroups} is the variety of algebras $(G,\cdot,/,\ld)$ satisfying the identities
$$ab\cdot c=cb\cdot a,\quad a\cdot bc=b\cdot ac,\quad a\cdot(a\ld b)=a\ld(a\cdot b)=(b/a)\cdot a=(b\cdot a)/a=b.$$

\emph{The variety of $\Z[x]/(x^2-1)$-modules} is the variety of algebras $(G,+,-,0,\varphi)$ such that $(G,+,-,0)$ is an abelian group and
$$\varphi(a+b)=\varphi(a)+\varphi(b),\quad \varphi(\varphi(a))=a.$$
(Clearly, the action of the ring $\Z[x]/(x^2-1)$ is given by the action of the monomial $x$, which in turn acts as an involutory automorphism $\varphi$.)

\begin{theorem}\label{Th:main_varieties}
The varieties of AG-groups, AG$^{**}$-quasigroups and $\Z[x]/(x^2-1)$-modules are term equivalent.
\end{theorem}

\begin{proof}
First, we translate AG-groups to AG$^{**}$-quasigroups. The identity $a\cdot bc=b\cdot ac$ was proved in Lemma \ref{L:paramedial}(3). It remains to show that AG-groups have term definable unique left and right division. We start with uniqueness of right division. If $xa=b$, then $xa\cdot a^{-1}=ba^{-1}$, and since 
$xa\cdot a^{-1}=a^{-1}a\cdot x=ex=x$, we necessarily have $x=ba^{-1}$. Now it is easy to check that this really is a solution: $ba^{-1}\cdot a=aa^{-1}\cdot b=eb=b$. Hence, we have a term definable unique right division, namely $$b/a=ba^{-1}$$
(see also \cite[Lemma 2.3]{P2} for this observation).
Now we solve $ax=b$. Since $ax=ea\cdot x=xa\cdot e$, we have $ax=b$ iff $xa\cdot e=b$ iff $x=(b/e)/a=be^{-1}\cdot a^{-1}=be\cdot a^{-1}$. Hence, we have a term definable unique left division, namely $$a\ld b=be\cdot a^{-1}.$$

In the second step, we translate AG$^{**}$-quasigroups to $\Z[x]/(x^2-1)$-modules. First recall that $e=a/a$ is a term definable left unit (independent of the choice of $a$), see Lemma \ref{L:paramedial}(1). Let $$a+b=ae\cdot b,\quad -a=(ae)\ld e,\quad 0=e, \quad \varphi(a)=ae$$
(the definition of the group operations can be found in \cite[Theorem 2.9]{ARS} and \cite[Theorem 2.4]{P2}).
To check $(G,+,-,0)$ is an abelian group, calculate $a+b=ae\cdot b=be\cdot a=b+a$, $0+a=ee\cdot a=ea=a$, $a+(-a)=ae\cdot(ae)\ld e=e=0$, and 
$$a+(b+c)=ae\cdot(be\cdot c)=ae\cdot(ce\cdot b)=ce\cdot(ae\cdot b)=(ae\cdot b)e\cdot c=(a+b)+c.$$
To check that $\varphi$ is an involutory automorphism, calculate $\varphi^2(a)=ae\cdot e=ee\cdot a=ea=a$, and finally, 
$$\varphi(a+b)=(ae\cdot b)e=eb\cdot ae=b\cdot ae=(be\cdot e)\cdot ae=\varphi(b)+\varphi(a)=\varphi(a)+\varphi(b).$$

In the last step, we translate $\Z[x]/(x^2-1)$-modules to AG-groups. Let $$a\cdot b=\varphi(a)+b,\quad a^{-1}=\varphi(-a),\quad e=0.$$
For the Abel-Grassman identity, $$ab\cdot c=\varphi(\varphi(a)+b)+c=a+\varphi(b)+c=\varphi(\varphi(c)+b)+a=cb\cdot a,$$
the inverse properties follow from
$$a\cdot a^{-1}=\varphi(a)+\varphi(-a)=0=e, \quad a^{-1}\cdot a=\varphi(\varphi(-a))+a=0=e,$$
and $e$ is a left unit since $ea=\varphi(0)+a=a$.
\end{proof}

Theorem \ref{Th:main} is an immediate consequence of Lemma \ref{L:paramedial} and Theorem \ref{Th:main_varieties}. However, Theorem \ref{Th:main_varieties} implies a stronger connection between the four concepts: any property or feature that only depends on term operations is equivalent in all four settings (for instance, subalgebras, congruences, homomorphisms, and their properties). In particular, it allows to translate many properties of modules to the AG-group setting.

Naturally, a subset $H$ of an AG-group $G$ is a \emph{sub-AG-group} if $a\cdot b\in H$, $a^{-1}\in H$ and $e\in H$ for every $a,b\in H$. Using Theorem \ref{Th:main_varieties}, $H$ is a sub-AG-group if and only if $H$ is a submodule of the corresponding $\Z[x]/(x^2-1)$-module, i.e., if and only if $a+b=ae\cdot b\in H$, $-a=(ae)\ld e=ee\cdot(ae)^{-1}=(ae)^{-1}=a^{-1}e\in H$, $0=e\in H$ and $\varphi(a)=ae\in H$ for every $a,b\in H$. Hence, for example, all finite AG-groups satisfy the \emph{Lagrange property} (if $B$ is a subalgebra of $A$, then $|B|$ divides $|A|$), because all finite modules do (see \cite{MK1} for a different argument).

Similarly for congruences. In modules, subalgebras and ideals (i.e., kernels of homomorphisms) is the same thing. Therefore, so it is in AG-groups. Every congruence $\alpha$ of an AG-group $G$ is uniquely determined by its block $e/\alpha$ containing the left unit $e$, and this block forms a subalgebra, i.e., a sub-AG-group.
% (with respect to operations $\cdot$, $^{-1}$, $e$). 
Conversely, every sub-AG-group $H$ of $G$ determines a unique congruence defined by $a\sim b$ iff $a-b\in H$.
%, where $-$ is the corresponding module operation. 
Since $a-b=a+(-b)=ae\cdot b^{-1}e=b^{-1}a$ (using paramediality), a sub-AG-group $H$ determines the congruence defined by $$a\sim b\quad\Leftrightarrow\quad b^{-1}a\in H.$$
In particular, all AG-groups have \emph{ideal-determined congruences} (congruences are uniquely determined by their block containing the left unit), and thus are \emph{congruence permutable} ($\alpha\vee\beta=\alpha\circ\beta=\beta\circ\alpha$ for any pair of congruences) and thus also \emph{congruence modular} (see \cite{GU} for more information on ideal-determined varieties). %Also, AG-groups are \emph{abelian} in the sense of commutator theory (the commutator satsifies $[1,1]=0$).

Let us note that a recent paper by Proti\'c \cite[Theorem 3.4]{P2} contains a more complicated description of the correspondence between subalgebras and congruences. He considers \emph{normal sub-AG-groups} as subsets $H\subseteq G$ closed with respect to all three AG-group operations and such that $u\cdot au^{-1}\in H$ for every $a\in H$ and $u\in G$. However, it is easy to see that every sub-AG-group is normal in Proti\'c's sense: we have $u\cdot au^{-1}=a\cdot uu^{-1}=ae\in H$, because both $a,e\in H$.

\section{Enumeration}

In \cite{SGS}, the enumeration of AG-groups of size $\leq11$ was presented, using a brute force search with a few non-trivial symmetry-breaking heuristics. Here we show how the algebraic theory can be used to obtain enumeration of much larger scale.

Let $(G,+)$ be an abelian group and $\varphi$ its involutory automorphism. The AG-group $(G,\cdot)$ with $x\cdot y=\varphi(x)+y$ will be denoted $\AG(G,\varphi)$. Theorem \ref{Th:main} says that every AG-group admits such a representation.
The first step in our enumeration is an isomorphism theorem.

\begin{proposition}\label{prop:iso}
Let $G,H$ be two abelian groups and $\varphi,\psi$ their involutory automorphisms, respectively. A mapping $f:G\to H$ is an isomorphism $\AG(G,\varphi)\simeq \AG(H,\psi)$ if and only if $f$ is a group isomorphism $G\simeq H$ and $\psi=f\varphi f^{-1}$.
\end{proposition}

\begin{proof}
It follows from Theorem \ref{Th:main_varieties} that a mapping $f:G\to H$ is a homomorphism with respect to AG-group operations, i.e., $f(a\cdot b)=f(a)\cdot f(b)$, $f(a^{-1})=f(a)^{-1}$ and $f(e)=e$, if and only if it is a homomorphism with respect to the corresponding $\Z[x]/(x^2-1)$-module operations, i.e., $f(a+b)=f(a)+f(b)$ and $f(\varphi(a))=\psi(f(a))$. The statement follows immediately.
%$(\Rightarrow)$ 
%Assume $f$ is an AG-group isomorphism, i.e., $$f(\varphi(x)+y)=f(x\cdot y)=f(x)\cdot f(y)=\psi f(x)+f(y)$$ for every $x,y\in G$. With $x=y=0$, we obtain
%$f(0)=f(0)+f(0)$, hence $f(0)=0$. With $y=0$, we then obtain $f\varphi(x)=\psi f(x)+f(0)=\psi f(x)$, hence $\psi=f\varphi f^{-1}$. Finally, with $x=\varphi^{-1}(u)$, we obtain $f(u+y)=\psi f \varphi^{-1}(u)+f(y)=f(u)+f(y)$, and we see that $f$ is a group isomorphism.
%
%$(\Leftarrow)$ 
%Assuming $f$ is a group isomorphism, we have $$f(x\cdot y)=f(\varphi(x)+y)=f\varphi(x)+f(y)=\psi f(x)+f(y)=f(x)\cdot f(y)$$ 
%for every $x,y\in G$, where the third step uses the identity $f\varphi=\psi f$.
\end{proof}

Let $a(n)$ denote the number of isomorphism classes of AG-groups with $n$ elements. According to Proposition \ref{prop:iso}, 
$a(n)$ equals the sum of the numbers of involutory automorphisms up to conjugacy, where the sum runs over all abelian groups of order $n$ up to isomorphism.
If $m$ and $n$ are coprime, the classification of finite abelian groups implies that $a(mn)=a(m)a(n)$.
Therefore, we can focus on $a(p^d)$ for prime powers $p^d$.
We start with the enumeration of AG-groups of prime order or prime squared order. 

\begin{proposition}\label{prop:enumeration} 
If $p>2$ is a prime, then $a(p)=2$ and $a(p^2)=5$. For $p=2$, $a(2)=1$ and $a(4)=4$.
\end{proposition}

\begin{proof}
First consider the prime size. In this case any AG-group is isomorphic to $\AG(\mathbb Z_p,k)$ where $k\in\mathbb Z_p^*$ satisfies $k^2\equiv1\pmod p$. Since $\mathbb Z_p^*$ is cyclic, there is only one element of order 2, hence the only solutions are $k=1$ and $k=p-1$. If $p=2$, this is only one solution.

For prime squared, there are two possibilities. If $G=\mathbb Z_{p^2}$, we proceed similarly. If $k^2\equiv1\pmod{p^2}$, that is, $p^2\mid k^2-1=(k-1)(k+1)$,
then either $k=1$, or $k=p^2-1$, or $p$ divides both $k-1$, $k+1$, which is impossible unless $p=2$. In any case, we obtain two solutions.

Let $G=(\mathbb Z_p)^2$. Its automorphism group is $GL(2,\mathbb F_p)$, hence we need to determine the number of conjugacy classes of matrices $A$ satisfying $A^2=I$. Since the polynomial $f=x^2-1$ splits over $\mathbb F_p$, such matrices are determined by their Jordan normal form. Now, the key observation is that if a matrix $A$ satisfies $f(A)=0$, then every eigenvalue of $A$ is a root of $f$. Hence, for $p = 2$, there are two possible Jordan forms
\[
  \begin{pmatrix}
    1 & 0 \\
     0 & 1 \\
  \end{pmatrix},\quad
  \begin{pmatrix}
    1 & 1 \\
     0 & 1 \\
  \end{pmatrix}.
\]
Both matrices satisfy $A^2=I$, hence $a(4)=2+2=4$. For $p>2$, there are five possible Jordan forms
\[
  \begin{pmatrix}
    1 & 0 \\
     0 & 1 \\
  \end{pmatrix},\quad
  \begin{pmatrix}
    -1 & 0 \\
     0 & -1 \\
  \end{pmatrix},\quad
  \begin{pmatrix}
    1 & 0 \\
     0 & -1 \\
  \end{pmatrix},\quad
  \begin{pmatrix}
    1 & 1 \\
     0 & 1 \\
  \end{pmatrix},\quad
  \begin{pmatrix}
    -1 & 1 \\
     0 & -1 \\
  \end{pmatrix}.
\]
The former three matrices satisfy $A^2=I$, while the latter two matrices do not. Hence $a(p^2)=2+3=5$.
\end{proof}

\begin{table} %[!hb]
\begin{center}
\begin{tabular}{|c|cccccccc|}\hline
$d$      & 1 & 2 & 3  & 4  & 5  & 6 & 7 & 8\\\hline
$a(2^d)$ & 1 & 4 & 10 & 29 & 69 & 187 & 449 & 1141  \\
$a(3^d)$ & 2 & 5 & 10 & 20 & 36 & 65 && \\
$a(5^d)$ & 2 & 5 & 10 & 20 &&&&  \\
$a(7^d)$ & 2 & 5 & 10 & &&&&  \\
$a(p^d)$ & 2 & 5 &&&&&&  \\\hline
\end{tabular}
\end{center}
%\smallskip
\caption{Enumeration of AG-groups.}\label{tab:enumeration}
\end{table} 

Further values of $a(p^d)$ can be evaluated in GAP \cite{GAP} by a straightforward calculation using Proposition \ref{prop:iso}. The values are summarized in Table \ref{tab:enumeration}.

\section{Beyond quasigroups?}

Theorem \ref{Th:main} translates all questions about AG-groups into questions about abelian groups and their involutory automorphisms (or, about $\Z[x]/(x^2-1)$-modules), and explains why properties of AG-groups resemble those of abelian groups --- recall the comment at the end of the first paragraph of the paper.
(It also shows that the claim in \cite{SA} that AG-groups are ``midway between quasigroups and abelian groups'' is not quite precise.)

A natural question for further research is, what about other classes of AG$^{**}$-groupoids? In particular, we have in mind the class of completely inverse AG$^{**}$-groupoids studied in detail by Dudek and Gigo\'n \cite{DG1,DG2}. For instance, their characterization of congruences is very similar to the one given in the theory of inverse semigroups. Is there an equivalence with some variety of inverse semigroups with operators? Perhaps, there is a nice semilinear representation, where the semimodule is built over a commutative inverse semigroup.

Some hints can be found in literature. In \cite[Section 3]{JK} and a few later papers, various kinds of (semi-)linear and (semi-)affine representations of medial groupoids are established. Perhaps some of the ideas can be exploited in the setting of completely inverse AG$^{**}$-groupoids. See also \cite[Theorems 2.7--2.10]{ARS} for an attempt to establish such a connection.


\begin{thebibliography}{99}


\bibitem{ARS}
I. Ahmad, M. Rashad, M. Shah, \emph{Constructions of some algebraic structures from each other.} Int. Math. Forum 7, No. 53-56, 2759-2766 (2012).
% preklad do comm. smg.

\bibitem{CJK}
J. R. Cho, J. Je\v zek, T. Kepka, \emph{Paramedial groupoids.} Czech. Math. J. 49, No.2, 277--290 (1999).

\bibitem{DG1}
W. Dudek, R. Gigo\'n, \emph{Congruences on completely inverse AG$^{**}$-groupoids}, Quasigroups Relat. Syst. 20, No. 2, 203--209 (2012).

\bibitem{DG2}
W. Dudek, R. Gigo\'n, \emph{Completely inverse AG$^{**}$-groupoids}, Semigroup Forum 87, No. 1, 201--229 (2013).

\bibitem{GAP}
The GAP Group, \emph{GAP -- Groups, Algorithms, and Programming, Version 4.5.7}; 2012. \texttt{(http://www.gap-system.org)}

\bibitem{GU}
H. P. Gumm, A. Ursini, \emph{Ideals in universal algebras}, Algebra Universalis 19 (1984), 45--54.

\bibitem{JK}
J. Je\v zek, T. Kepka. \emph{Medial groupoids.} Rozpravy \v CSAV, Rada mat. a p\v rir. v\v ed 93/2, 1983, 93 pp.

\bibitem{K}
M. S. Kamran, \emph{Conditions for LA-semigroups to resemble associative structures}, Ph.D. Thesis, Quaid-i-Azam Universiti, Islamabad, 1993.

\bibitem{KN}
P. N\v emec, T. Kepka, \emph{T-quasigroups I} , Acta Univ. Carolin. Math. Phys. 12 (1971), no. 1, 39--49.

\bibitem{KS}
M. Khan, S. Anis, \emph{An analogy of Clifford decomposition theorem for Abel-Grassmann groupoids.} Algebra Colloq. 21 (2014), No. 2, 347-353.

\bibitem{MK1}
Q. Mushtaq, M.S. Kamran, \emph{On left almost groups}, Proc. Pak. Acad. of Sciences, 33 (1996), 1--2.
% zrejme prvni pokus o AG-grupy krome te thesis

%\bibitem{MK2}
%Q. Mushtaq, M. S. Kamran, \emph{Finite AG-groupoid with left identity and left zero}, Int. J. Math. Math. Sci. 27, 387--389 (2001).

\bibitem{MK3}
Q. Mushtaq, M. Khan, \emph{Semilattice decomposition of locally associative AG$^{**}$-groupoids}, Algebra Colloq. 16 (2009), 17--22.

\bibitem{P1}
P. Proti\'c, \emph{Congruencies on an inverse AG**-groupoid via the natural partial order.} Quasigroups Relat. Syst. 17 (2009), No. 2, 283-290.

\bibitem{P2}
P. Proti\'c, \emph{Some remarks on Abel-Grassmann's groups}, Quasigroups and Related Systems 20 (2012), 267--274.

%P. Proti\'c, N. Stevanovi\'c, \emph{On Abel-Grassmann's groupoids (exposition)}, Proc. Math. Conf. Pristina, 27--29 (1994)

\bibitem{SA}
M. Shah, A. Ali, \emph{Some structural properties of AG-groups.} Int. Math. Forum 6 (2011), No. 33-36, 1661-1667.
% naprosta pitomost

\bibitem{SGS}
M. Shah, C. Gretton, V. Sorge, \emph{Enumerating AG-groups with a study of Smaradache AG-groups.} Int. Math. Forum 6 (2011), No. 61-64, 3079-3086.
% enumerace

\bibitem{T}
K. Toyoda, \emph{On axioms of linear functions.} Proc. Imp. Acad. Tokyo 17 (1941), 221--227.

\end{thebibliography}
\end{document}